\documentclass[12pt]{article}

\usepackage[T1]{fontenc}
\usepackage{microtype}
\usepackage[numbers,sort&compress]{natbib}
\usepackage{amsmath,amssymb,amsfonts,amsthm,mathtools}
\usepackage{geometry}
\usepackage{xcolor}
\usepackage{hyperref}

\definecolor{Emerald}{HTML}{00674F}
\hypersetup{
    colorlinks=true,
    linkcolor=blue,
    citecolor=Emerald,
    urlcolor=blue,
    pdftitle={An Exact Tail Condition for the Largest Error in Estimating a Rank-One Direction},
    pdfauthor={Guilherme Vianna}
}

\newtheorem{theorem}{Theorem}[section]
\newtheorem{lemma}[theorem]{Lemma}
\newtheorem{remark}[theorem]{Remark}
\newtheorem{corollary}[theorem]{Corollary}
\newcommand{\E}{\mathbb{E}}
\renewcommand{\Pr}{\mathbb{P}}
\renewcommand{\leq}{\leqslant}
\renewcommand{\geq}{\geqslant}

\newcommand{\R}{\mathbb{R}}
\newcommand{\1}{\mathbf{1}}
\newcommand{\op}{\mathrm{op}}

\title{An Exact Tail Condition for the Largest Error in Estimating a Rank-One Direction}
\author{
Guilherme Vianna\\[3pt]
\small University of S\~ao Paulo, S\~ao Paulo, Brazil, and Columbia University, New York, NY, USA\\
\small \texttt{guilherme.dias.vianna@usp.br}
}
\date{}

\begin{document}

\maketitle

\begin{abstract}
We consider a rectangular random matrix formed by adding a rank-one term to a matrix with independent entries. The direction on the left side of that term is estimated by the leading left singular vector, and the error is multiplied by the square of the size of the added term. For entries with mean zero and variance one, we identify the exact tail condition for the following statement to hold for every deterministic sequence of directions: the largest error over all sizes of the added term converges to the same fixed value determined by the limiting ratio of rows to columns. This condition (weaker than the existence of fourth moments) requires that the tail probability, multiplied by the fourth power of the threshold, to tend to zero. If it fails, convergence to this value already fails when both directions are coordinate vectors, even at a size fixed before the matrix is drawn. For regularly varying tails of order between two and four, the largest error tends to infinity. 
\end{abstract}

\noindent\textbf{Keywords:} random matrices; rank-one matrices; singular vectors; tail conditions; direction estimation.

\medskip
\noindent\textbf{Mathematics Subject Classification 2020:} 60B20; 15A18; 62H25.

\section{Introduction}
\label{sec:introduction}

Leading singular vectors are used to estimate the main directions in a data matrix. This is a basic step in principal component analysis, low-rank matrix denoising, large factor models, and signal recovery. In these problems, replacing an unknown direction by the leading singular vector affects both the estimated component and the reconstructed matrix; for a rank-one matrix, the error studied below is also the squared reconstruction error left after projection onto the estimated direction. See \cite{Jolliffe2002,Ding2020,DingYang2018,GavishDonoho2017,Onatski2009} for representative statistical settings.

We study the simplest rectangular model in which a rank-one matrix is added to a matrix with independent entries. Existing results describe the leading singular value and singular vectors when the size of the added term is fixed \cite[Theorems~2.2--2.3, pp.~392--393]{Ding2020}, while \cite[Theorem~2.7, p.~1687]{DingYang2018} identifies the exact tail condition governing the largest singular value of the noise matrix. These results do not by themselves control the direction error simultaneously over every possible size of the rank-one term.

First, we show that the same fourth-order tail condition is sufficient for the largest error to converge, for every deterministic sequence of directions, to a value depending only on the limiting ratio of the dimensions. Next, we show that the condition is necessary for this statement over all deterministic direction sequences: if it fails, the convergence already fails for coordinate directions at a fixed size. Under regularly varying tails of order between two and four, the largest error diverges. Finally, for Gaussian noise, we state the result in the original matrix scale and show that the leading constant cannot be reduced. The new parts of the proof are the passage from fixed sizes to all sizes, the direct treatment of large sizes, and the construction based on a single large noise entry.

\section{Preliminaries}
\label{sec:preliminaries}

\subsection{Matrix model and assumptions}

Let \(N=N_m\) and \(m\) tend to infinity in such a way that
\begin{equation}
 c_{N,m}:=\frac{N}{m}\rightarrow c\in\left(0,\infty\right).
 \label{eq:dimension-ratio}
\end{equation}
Let \(\Xi\) be a real random variable satisfying \(\E\left[\Xi\right]=0\) and \(\E\left[\Xi^2\right]=1\). Let \(\boldsymbol{Z}=\left(\Xi_{ij}\right)\) be an \(N\times m\) matrix with independent entries distributed as \(\Xi\). The simplest choice for the added matrix is \(\theta\boldsymbol{e}_1^{\left(N\right)}\left(\boldsymbol{e}_1^{\left(m\right)}\right)^{\top}\), which changes only the upper-left entry. More generally, for deterministic unit vectors \(\boldsymbol{u}\in\R^N\) and \(\boldsymbol{v}\in\R^m\), define
\[
 \boldsymbol{X}:=\frac{\boldsymbol{Z}}{\sqrt{m}},
 \qquad
 \boldsymbol{A}\left(\theta\right):=\boldsymbol{X}+\theta\boldsymbol{u}\boldsymbol{v}^{\top},
 \qquad
 \theta\geq0.
\]
The vectors \(\boldsymbol{u}\) and \(\boldsymbol{v}\) determine the rank-one matrix, and \(\theta\) gives its size.

\subsection{Singular vectors and the error}

For a matrix \(\boldsymbol{B}\), let \(\sigma_j\left(\boldsymbol{B}\right)\) denote its \(j\)-th largest singular value and let \(\left\|\boldsymbol{B}\right\|_{\op}:=\sigma_1\left(\boldsymbol{B}\right)\).

To remove any ambiguity when the largest singular value is repeated, use the same deterministic tie-breaking rule for every matrix: project the coordinate vectors onto the leading left singular subspace, take the first nonzero projection, and normalize it. Let \(\boldsymbol{\hat{u}}\left(\theta\right)\) be the vector returned by this rule for \(\boldsymbol{A}\left(\theta\right)\). The rule is unchanged when a matrix is multiplied by a positive number and is Borel measurable. The resulting vector is jointly measurable in the entries of \(\boldsymbol{Z}\) and in \(\theta\); hence the suprema below are measurable after completing the probability space. We measure the error by
\[
 \mathcal{L}_{N,m}\left(\theta\right)
 :=\theta^2\left[1-\left|\boldsymbol{u}^{\top}\boldsymbol{\hat{u}}\left(\theta\right)\right|^2\right].
\]
The quantity in brackets is the squared norm of the part of \(\boldsymbol{\hat{u}}\left(\theta\right)\) perpendicular to \(\boldsymbol{u}\). Thus \(\mathcal{L}_{N,m}\left(\theta\right)\) records this error after multiplication by \(\theta^2\).

The main result compares this error with the following function. For \(\gamma>0\), define
\[
 \mathcal{R}_{\gamma}\left(\theta\right)
 :=
 \begin{cases}
 \theta^2, & 0\leq\theta\leq\gamma^{1/4},\\[5pt]
 \displaystyle\frac{\gamma\left(1+\theta^2\right)}{\gamma+\theta^2}, & \theta>\gamma^{1/4}.
 \end{cases}
\]

\section{Main results}
\label{sec:main-result}

We now state the main result. Its sufficient part concerns every deterministic sequence of directions, while its converse shows that this statement fails for a particular sequence of coordinate directions when the tail condition is violated.

\begin{theorem}
\label{thm:exact-tail-condition}
Suppose that
\begin{equation}
 \lim_{x\to\infty}x^4\Pr\left\{\left|\Xi\right|>x\right\}=0.
 \label{eq:fourth-tail-condition}
\end{equation}
Then, for every deterministic sequence of unit vectors \(\boldsymbol{u}\in\R^N\) and \(\boldsymbol{v}\in\R^m\),
\begin{equation}
 \sup_{\theta\geq0}\left\{\mathcal{L}_{N,m}\left(\theta\right)-\mathcal{R}_{c_{N,m}}\left(\theta\right)\right\}
 \xrightarrow{\Pr}0.
 \label{eq:uniform-upper-curve}
\end{equation}
Consequently,
\begin{equation}
 \sup_{\theta\geq0}\mathcal{L}_{N,m}\left(\theta\right)
 \xrightarrow{\Pr}c\vee\sqrt{c}.
 \label{eq:largest-error}
\end{equation}

Conversely, suppose that \eqref{eq:fourth-tail-condition} fails. For every \(c\in\left(0,\infty\right)\), take \(N_m=\left\lfloor cm\right\rfloor\), \(\boldsymbol{u}=\boldsymbol{e}_1^{\left(N_m\right)}\), and \(\boldsymbol{v}=\boldsymbol{e}_1^{\left(m\right)}\). Then \eqref{eq:largest-error} fails. More precisely, there exist \(\varepsilon>0\), a fixed \(\theta_0>0\), and a subsequence \(m_k\) such that
\[
 \liminf_{k\to\infty}\Pr\left\{\mathcal{L}_{N_{m_k},m_k}\left(\theta_0\right)>c\vee\sqrt{c}+\varepsilon\right\}>0.
\]
Thus the failure occurs even when \(\theta_0\) is fixed before the matrix is drawn.

If, in addition,
\[
 \Pr\left\{\left|\Xi\right|>x\right\}=x^{-\alpha}\ell\left(x\right),
 \qquad
 2<\alpha<4,
\]
for a slowly varying function \(\ell\), then
\[
 \sup_{\theta\geq0}\mathcal{L}_{N,m}\left(\theta\right)
 \xrightarrow{\Pr}\infty.
\]
\end{theorem}

\begin{remark}
Finite fourth moment implies \eqref{eq:fourth-tail-condition}, because
\[
 x^4\Pr\left\{\left|\Xi\right|>x\right\}
 \leq\E\left[\left|\Xi\right|^4\1\left\{\left|\Xi\right|>x\right\}\right]
 \rightarrow0.
\]
Condition \eqref{eq:fourth-tail-condition} is weaker than the existence of a finite fourth moment.
\end{remark}

\begin{remark}
The condition is exact for the statement that \eqref{eq:largest-error} holds for every deterministic sequence of directions. To disprove that statement when the condition fails, it is enough to give one sequence, and the theorem uses the two coordinate directions.
\end{remark}

\subsection{A fixed value of \texorpdfstring{\(\theta\)}{theta}}
\label{subsec:fixed-theta}

For \(\theta>c^{1/4}\), define
\[
 \rho_c\left(\theta\right)
 :=\frac{\sqrt{\left(1+\theta^2\right)\left(c+\theta^2\right)}}{\theta},
 \qquad
 a_c\left(\theta\right)
 :=\frac{\theta^4-c}{\theta^2\left(\theta^2+c\right)}.
\]
These functions satisfy
\[
 \theta^2\left[1-a_c\left(\theta\right)\right]
 =\frac{c\left(1+\theta^2\right)}{c+\theta^2}
 =\mathcal{R}_c\left(\theta\right).
\]

\begin{lemma}
\label{lem:fixed-theta}
Under \eqref{eq:fourth-tail-condition},
\[
 \left\|\boldsymbol{X}\right\|_{\op}\xrightarrow{\Pr}1+\sqrt{c}.
\]
For every fixed \(\theta>c^{1/4}\),
\[
 \sigma_1\left(\boldsymbol{A}\left(\theta\right)\right)\xrightarrow{\Pr}\rho_c\left(\theta\right),
 \qquad
 \sigma_2\left(\boldsymbol{A}\left(\theta\right)\right)\xrightarrow{\Pr}1+\sqrt{c},
\]
and
\[
 \left|\boldsymbol{u}^{\top}\boldsymbol{\hat{u}}\left(\theta\right)\right|^2
 \xrightarrow{\Pr}a_c\left(\theta\right).
\]
Consequently, \(\mathcal{L}_{N,m}\left(\theta\right)\xrightarrow{\Pr}\mathcal{R}_c\left(\theta\right)\).
\end{lemma}

\begin{proof}
The first conclusion follows by applying \cite[Theorem~2.7, p.~1687]{DingYang2018} with the identity population matrix. Under the present normalization, its tail assumption is \eqref{eq:fourth-tail-condition}, and the largest eigenvalue of \(\boldsymbol{X}\boldsymbol{X}^{\top}\) converges to \(\left(1+\sqrt{c}\right)^2\). Taking square roots gives \(\sigma_1\left(\boldsymbol{X}\right)\xrightarrow{\Pr}1+\sqrt{c}\).

We next reduce the remaining conclusions to bounded entries. For \(K>0\), let
\[
 \mu_K:=\E\left[\Xi\1\left\{\left|\Xi\right|\leq K\right\}\right],
 \qquad
 s_K^2:=\operatorname{Var}\left(\Xi\1\left\{\left|\Xi\right|\leq K\right\}\right),
\]
and, for all sufficiently large \(K\), set
\[
 \Xi^{\left(K\right)}
 :=\frac{\Xi\1\left\{\left|\Xi\right|\leq K\right\}-\mu_K}{s_K}.
\]
Then \(\Xi^{\left(K\right)}\) is centered, has variance one, is bounded, and converges to \(\Xi\) in \(L^2\) as \(K\to\infty\).

Apply this truncation and rescaling to every entry of \(\boldsymbol{Z}\), and let \(\boldsymbol{X}^{\left(K\right)}\) be the resulting matrix divided by \(\sqrt{m}\). Put \(\boldsymbol{D}^{\left(K\right)}:=\boldsymbol{X}-\boldsymbol{X}^{\left(K\right)}\). The entries of \(\sqrt{m}\boldsymbol{D}^{\left(K\right)}\) are independent copies of
\[
 \Delta_K:=\Xi-\Xi^{\left(K\right)},
 \qquad
 \E\left[\Delta_K\right]=0,
 \qquad
 \tau_K^2:=\E\left[\Delta_K^2\right]\rightarrow0.
\]
If \(\tau_K=0\), then \(\boldsymbol{D}^{\left(K\right)}=\boldsymbol{0}\) almost surely. If \(\tau_K>0\), the random variable \(\Delta_K/\tau_K\) also satisfies \eqref{eq:fourth-tail-condition}. Indeed, \(\Xi^{\left(K\right)}\) is bounded, so, for all sufficiently large \(x\),
\[
 \Pr\left\{\left|\Delta_K\right|>x\right\}
 \leq\Pr\left\{\left|\Xi\right|>x/2\right\}.
\]
Theorem~2.7 in \cite{DingYang2018} therefore gives, in either case,
\[
 \left\|\boldsymbol{D}^{\left(K\right)}\right\|_{\op}
 \xrightarrow{\Pr}\tau_K\left(1+\sqrt{c}\right)
\]
for every fixed \(K\). Consequently, for every \(\eta>0\),
\begin{equation}
 \lim_{K\to\infty}\limsup_{N,m\to\infty}
 \Pr\left\{\left\|\boldsymbol{X}-\boldsymbol{X}^{\left(K\right)}\right\|_{\op}>\eta\right\}=0.
 \label{eq:operator-approximation}
\end{equation}

Since \(\Xi^{\left(K\right)}\) is bounded, the assumptions of \cite[Theorems~2.2--2.3, pp.~392--393]{Ding2020} hold for
\[
 \boldsymbol{A}^{\left(K\right)}\left(\theta\right)
 :=\boldsymbol{X}^{\left(K\right)}+\theta\boldsymbol{u}\boldsymbol{v}^{\top}.
\]
To translate the notation, write Ding's dimensions as \(M_{\mathrm D}\times N_{\mathrm D}\), with entries normalized to have variance \(1/N_{\mathrm D}\), and let \(c_{\mathrm D}:=N_{\mathrm D}/M_{\mathrm D}\). Here \(M_{\mathrm D}=N\), \(N_{\mathrm D}=m\), and therefore \(c_{\mathrm D}=c_{N,m}^{-1}\); the added singular value in that paper is \(d=\theta\). With this substitution, the outlier location in equation~(2.6) of that paper converges to \(\rho_c\left(\theta\right)^2\), and the left singular-vector expression in equation~(2.9) becomes \(a_c\left(\theta\right)\). Because \(\theta>c^{1/4}\) is fixed, it remains above \(c_{\mathrm D}^{-1/4}\) for all sufficiently large dimensions, and the separation condition between different added singular values is automatic in rank one. Therefore,
\[
 \sigma_1\left(\boldsymbol{A}^{\left(K\right)}\left(\theta\right)\right)
 \xrightarrow{\Pr}\rho_c\left(\theta\right)
\]
and, for a leading left singular vector \(\boldsymbol{\hat{u}}^{\left(K\right)}\left(\theta\right)\),
\[
 \left|\boldsymbol{u}^{\top}\boldsymbol{\hat{u}}^{\left(K\right)}\left(\theta\right)\right|^2
 \xrightarrow{\Pr}a_c\left(\theta\right).
\]

The second singular value requires a separate argument. Applied to the eigenvalues of \(\boldsymbol{X}^{\left(K\right)}\left(\boldsymbol{X}^{\left(K\right)}\right)^{\top}\), \cite[Lemma~4.12, p.~405]{Ding2020}, together with the fixed-index classical locations defined on that page, gives, after taking square roots,
\[
 \sigma_j\left(\boldsymbol{X}^{\left(K\right)}\right)
 \xrightarrow{\Pr}1+\sqrt{c},
 \qquad
 j=1,2,3.
\]
For all sufficiently large dimensions, \(\boldsymbol{X}^{\left(K\right)}\) has at least three singular values. Since \(\boldsymbol{A}^{\left(K\right)}\left(\theta\right)-\boldsymbol{X}^{\left(K\right)}\) has rank one, singular-value interlacing gives
\[
 \sigma_3\left(\boldsymbol{X}^{\left(K\right)}\right)
 \leq
 \sigma_2\left(\boldsymbol{A}^{\left(K\right)}\left(\theta\right)\right)
 \leq
 \sigma_1\left(\boldsymbol{X}^{\left(K\right)}\right).
\]
Consequently,
\[
 \sigma_2\left(\boldsymbol{A}^{\left(K\right)}\left(\theta\right)\right)
 \xrightarrow{\Pr}1+\sqrt{c}.
\]
The limiting first singular value is strictly larger, since
\[
 \rho_c\left(\theta\right)^2-\left(1+\sqrt{c}\right)^2
 =\frac{\left(\theta^2-\sqrt{c}\right)^2}{\theta^2}>0.
\]

Let
\[
 d_{c,\theta}:=\rho_c\left(\theta\right)-\left(1+\sqrt{c}\right)>0.
\]
Fix \(0<\eta<d_{c,\theta}/4\). By \eqref{eq:operator-approximation}, \(K\) can be chosen so large that
\[
 \limsup_{N,m\to\infty}
 \Pr\left\{\left\|\boldsymbol{X}-\boldsymbol{X}^{\left(K\right)}\right\|_{\op}>\eta\right\}
\]
is arbitrarily small. For this fixed \(K\), the first two singular values of \(\boldsymbol{A}^{\left(K\right)}\left(\theta\right)\) differ by more than \(d_{c,\theta}/2\) with probability tending to one. Weyl's inequality shows that the corresponding singular values of \(\boldsymbol{A}\left(\theta\right)\) are within \(\eta\) of them. Wedin's inequality then gives, on the intersection of these events,
\[
 \min_{\varsigma\in\left\{-1,1\right\}}
 \left\|\boldsymbol{\hat{u}}\left(\theta\right)-\varsigma\boldsymbol{\hat{u}}^{\left(K\right)}\left(\theta\right)\right\|
 \leq C_{c,\theta}\left\|\boldsymbol{X}-\boldsymbol{X}^{\left(K\right)}\right\|_{\op}.
\]
First let \(N,m\to\infty\) with \(K\) fixed, and then let \(K\to\infty\). Weyl's inequality gives the two singular-value limits for \(\boldsymbol{A}\left(\theta\right)\), and the last display gives the stated limit for the square of the inner product. The conclusion for \(\mathcal{L}_{N,m}\left(\theta\right)\) follows from its definition.
\end{proof}

\subsection{From fixed values to all values of \texorpdfstring{\(\theta\)}{theta}}
\label{subsec:all-theta}

\begin{lemma}
\label{lem:bounded-uniformity}
For every closed bounded interval \(I\subset\left(c^{1/4},\infty\right)\),
\[
 \sup_{\theta\in I}\left|\mathcal{L}_{N,m}\left(\theta\right)-\mathcal{R}_{c_{N,m}}\left(\theta\right)\right|
 \xrightarrow{\Pr}0.
\]
\end{lemma}

\begin{proof}
Write \(I=\left[a,b\right]\). The function
\[
 g_c\left(\theta\right):=\rho_c\left(\theta\right)-\left(1+\sqrt{c}\right)
\]
is continuous and strictly positive on \(I\). Hence \(g_*:=\inf_{\theta\in I}g_c\left(\theta\right)>0\).

Fix \(h<g_*/8\), and choose finitely many points \(\theta_1,\ldots,\theta_J\) such that every \(\theta\in I\) satisfies \(\left|\theta-\theta_j\right|\leq h\) for some \(j\). By Lemma~\ref{lem:fixed-theta}, with probability tending to one, all the selected points satisfy
\[
 \sigma_1\left(\boldsymbol{A}\left(\theta_j\right)\right)
 -\sigma_2\left(\boldsymbol{A}\left(\theta_j\right)\right)
 \geq\frac{g_*}{2}.
\]
For \(\left|\theta-\theta_j\right|\leq h\),
\[
 \left\|\boldsymbol{A}\left(\theta\right)-\boldsymbol{A}\left(\theta_j\right)\right\|_{\op}
 =\left|\theta-\theta_j\right|\leq h.
\]
Weyl's inequality therefore gives
\[
 \sigma_1\left(\boldsymbol{A}\left(\theta\right)\right)
 -\sigma_2\left(\boldsymbol{A}\left(\theta\right)\right)
 \geq\frac{g_*}{2}-2h
 \geq\frac{g_*}{4}.
\]
Wedin's inequality then gives, on the preceding event,
\[
 \min_{\varsigma\in\left\{-1,1\right\}}
 \left\|\boldsymbol{\hat{u}}\left(\theta\right)-\varsigma\boldsymbol{\hat{u}}\left(\theta_j\right)\right\|
 \leq C_I h
\]
uniformly over \(\theta\in I\), where \(C_I\) depends only on \(I\). Since \(\theta\) is bounded on \(I\),
\[
 \left|\mathcal{L}_{N,m}\left(\theta\right)-\mathcal{L}_{N,m}\left(\theta_j\right)\right|
 \leq C_I h.
\]
For all sufficiently large \(N,m\), the derivatives of \(\mathcal{R}_{c_{N,m}}\) are bounded on \(I\) by a constant depending only on \(I\). The mean value theorem therefore gives
\[
 \left|\mathcal{R}_{c_{N,m}}\left(\theta\right)-\mathcal{R}_{c_{N,m}}\left(\theta_j\right)\right|
 \leq C_I h.
\]
At the finitely many selected points, Lemma~\ref{lem:fixed-theta} and \eqref{eq:dimension-ratio} imply
\[
 \max_{1\leq j\leq J}
 \left|\mathcal{L}_{N,m}\left(\theta_j\right)-\mathcal{R}_{c_{N,m}}\left(\theta_j\right)\right|
 \xrightarrow{\Pr}0.
\]
Letting first \(N,m\to\infty\) and then \(h\downarrow0\) proves the claim.
\end{proof}

\begin{lemma}
\label{lem:large-theta}
Put
\[
 W:=\left\|\boldsymbol{X}\right\|_{\op},
 \qquad
 Q_{N,m}:=\left\|\left(\boldsymbol{I}_N-\boldsymbol{u}\boldsymbol{u}^{\top}\right)\boldsymbol{X}\boldsymbol{v}\right\|^2.
\]
There is a constant \(C\) such that, whenever \(\theta\geq3W\),
\begin{equation}
 \min_{\varsigma\in\left\{-1,1\right\}}
 \left\|\theta\left(\boldsymbol{I}_N-\boldsymbol{u}\boldsymbol{u}^{\top}\right)\boldsymbol{\hat{u}}\left(\theta\right)
 -\varsigma\left(\boldsymbol{I}_N-\boldsymbol{u}\boldsymbol{u}^{\top}\right)\boldsymbol{X}\boldsymbol{v}\right\|
 \leq\frac{CW^2}{\theta}.
 \label{eq:large-theta-vector}
\end{equation}
Consequently, on the event \(M\geq3W\),
\begin{equation}
 \sup_{\theta\geq M}\left|\mathcal{L}_{N,m}\left(\theta\right)-Q_{N,m}\right|
 \leq C\left(\frac{W^3}{M}+\frac{W^4}{M^2}\right).
 \label{eq:large-theta-error}
\end{equation}
Moreover,
\begin{equation}
 Q_{N,m}-c_{N,m}\xrightarrow{\Pr}0.
 \label{eq:q-convergence}
\end{equation}
\end{lemma}

\begin{proof}
Fix \(\theta\geq3W\), and write \(S=\sigma_1\left(\boldsymbol{A}\left(\theta\right)\right)\). Let \(\boldsymbol{\hat{u}}\left(\theta\right)\) be the vector selected above, and choose a corresponding unit right singular vector \(\boldsymbol{\hat{v}}\). Weyl's inequality gives
\[
 \left|S-\theta\right|\leq W,
 \qquad
 S\geq\theta-W\geq\frac{2\theta}{3}.
\]
The singular-vector equations imply
\[
 S\left(\boldsymbol{I}_N-\boldsymbol{u}\boldsymbol{u}^{\top}\right)\boldsymbol{\hat{u}}\left(\theta\right)
 =\left(\boldsymbol{I}_N-\boldsymbol{u}\boldsymbol{u}^{\top}\right)\boldsymbol{X}\boldsymbol{\hat{v}},
\]
and
\[
 S\left(\boldsymbol{I}_m-\boldsymbol{v}\boldsymbol{v}^{\top}\right)\boldsymbol{\hat{v}}
 =\left(\boldsymbol{I}_m-\boldsymbol{v}\boldsymbol{v}^{\top}\right)\boldsymbol{X}^{\top}\boldsymbol{\hat{u}}\left(\theta\right).
\]
Hence
\[
 \left\|\left(\boldsymbol{I}_N-\boldsymbol{u}\boldsymbol{u}^{\top}\right)\boldsymbol{\hat{u}}\left(\theta\right)\right\|\leq\frac{W}{S},
 \qquad
 \left\|\left(\boldsymbol{I}_m-\boldsymbol{v}\boldsymbol{v}^{\top}\right)\boldsymbol{\hat{v}}\right\|\leq\frac{W}{S}.
\]

Put \(U_{\theta}:=\boldsymbol{u}^{\top}\boldsymbol{\hat{u}}\left(\theta\right)\) and \(V_{\theta}:=\boldsymbol{v}^{\top}\boldsymbol{\hat{v}}\). Taking the inner product of the first singular-vector equation with \(\boldsymbol{\hat{u}}\left(\theta\right)\) gives
\[
 S=\theta U_{\theta}V_{\theta}+\boldsymbol{\hat{u}}\left(\theta\right)^{\top}\boldsymbol{X}\boldsymbol{\hat{v}}.
\]
Since \(S\geq\theta-W\) and \(\left|\boldsymbol{\hat{u}}\left(\theta\right)^{\top}\boldsymbol{X}\boldsymbol{\hat{v}}\right|\leq W\),
\[
 U_{\theta}V_{\theta}\geq1-\frac{2W}{\theta}>0.
\]
Thus there is a sign \(\varsigma\in\left\{-1,1\right\}\) such that
\[
 \widetilde{\boldsymbol{u}}:=\varsigma\boldsymbol{\hat{u}}\left(\theta\right),
 \qquad
 \widetilde{\boldsymbol{v}}:=\varsigma\boldsymbol{\hat{v}}
\]
satisfy \(\boldsymbol{u}^{\top}\widetilde{\boldsymbol{u}}\geq0\) and \(\boldsymbol{v}^{\top}\widetilde{\boldsymbol{v}}\geq0\). The second perpendicular-component bound then gives
\[
 \left\|\widetilde{\boldsymbol{v}}-\boldsymbol{v}\right\|^2
 =2\left(1-\boldsymbol{v}^{\top}\widetilde{\boldsymbol{v}}\right)
 \leq2\left[1-\left(\boldsymbol{v}^{\top}\widetilde{\boldsymbol{v}}\right)^2\right]
 \leq\frac{2W^2}{S^2}.
\]

Returning to the first projected equation, now with the oriented pair, gives
\[
 \begin{aligned}
 &\theta\left(\boldsymbol{I}_N-\boldsymbol{u}\boldsymbol{u}^{\top}\right)\widetilde{\boldsymbol{u}}
 -\left(\boldsymbol{I}_N-\boldsymbol{u}\boldsymbol{u}^{\top}\right)\boldsymbol{X}\boldsymbol{v}\\
 &\qquad=\left(\frac{\theta}{S}-1\right)
 \left(\boldsymbol{I}_N-\boldsymbol{u}\boldsymbol{u}^{\top}\right)\boldsymbol{X}\boldsymbol{v}
 +\frac{\theta}{S}\left(\boldsymbol{I}_N-\boldsymbol{u}\boldsymbol{u}^{\top}\right)\boldsymbol{X}\left(\widetilde{\boldsymbol{v}}-\boldsymbol{v}\right).
 \end{aligned}
\]
Using \(\left|S-\theta\right|\leq W\), \(S\geq2\theta/3\), and \(\left\|\widetilde{\boldsymbol{v}}-\boldsymbol{v}\right\|\leq\sqrt{2}W/S\), we obtain
\[
 \left\|\theta\left(\boldsymbol{I}_N-\boldsymbol{u}\boldsymbol{u}^{\top}\right)\widetilde{\boldsymbol{u}}
 -\left(\boldsymbol{I}_N-\boldsymbol{u}\boldsymbol{u}^{\top}\right)\boldsymbol{X}\boldsymbol{v}\right\|
 \leq\frac{CW^2}{\theta}.
\]
Since \(\widetilde{\boldsymbol{u}}=\varsigma\boldsymbol{\hat{u}}\left(\theta\right)\), this is equivalent to \eqref{eq:large-theta-vector}.

On the event \(M\geq3W\), apply \eqref{eq:large-theta-vector} to every \(\theta\geq M\). The two vectors in that bound have norms differing by at most \(CW^2/\theta\), while \(\left\|\left(\boldsymbol{I}_N-\boldsymbol{u}\boldsymbol{u}^{\top}\right)\boldsymbol{X}\boldsymbol{v}\right\|\leq W\). Since
\[
 \mathcal{L}_{N,m}\left(\theta\right)
 =\left\|\theta\left(\boldsymbol{I}_N-\boldsymbol{u}\boldsymbol{u}^{\top}\right)\boldsymbol{\hat{u}}\left(\theta\right)\right\|^2,
\]
it follows that
\[
 \left|\mathcal{L}_{N,m}\left(\theta\right)-Q_{N,m}\right|
 \leq C\left(\frac{W^3}{\theta}+\frac{W^4}{\theta^2}\right).
\]
Taking the supremum over \(\theta\geq M\) proves \eqref{eq:large-theta-error}.

It remains to prove \eqref{eq:q-convergence}. Define
\[
 Y_i:=\sum_{j=1}^m v_j\Xi_{ij},
 \qquad
 1\leq i\leq N.
\]
For each \(N,m\), the variables \(Y_1,\ldots,Y_N\) are independent and identically distributed, centered, and have variance one. We first show that their squares are uniformly integrable: large values contribute arbitrarily little to their expectations, uniformly over all dimensions and all deterministic unit vectors \(\boldsymbol{v}\).

Fix \(K>0\), and write
\[
 \Xi=H_K+R_K,
 \qquad
 H_K:=\Xi\1\left\{\left|\Xi\right|\leq K\right\}
 -\E\left[\Xi\1\left\{\left|\Xi\right|\leq K\right\}\right].
\]
Both \(H_K\) and \(R_K\) are centered, and \(\E\left[R_K^2\right]\to0\). Let \(H_{K,ij}\) and \(R_{K,ij}\) be the corresponding parts of \(\Xi_{ij}\), and define
\[
 A_{K,i}:=\sum_{j=1}^m v_jH_{K,ij},
 \qquad
 B_{K,i}:=\sum_{j=1}^m v_jR_{K,ij}.
\]
Then \(Y_i=A_{K,i}+B_{K,i}\), and independence across \(j\) gives
\[
 \sup_{m,\,\left\|\boldsymbol{v}\right\|=1}\E\left[A_{K,i}^4\right]
 \leq\E\left[H_K^4\right]+3\E\left[H_K^2\right]^2<\infty.
\]
Thus, for fixed \(K\), the squares \(A_{K,i}^2\) are uniformly integrable. Moreover,
\[
 \E\left[\left|Y_i^2-A_{K,i}^2\right|\right]
 \leq\left\|Y_i-A_{K,i}\right\|_2\left\|Y_i+A_{K,i}\right\|_2
 \leq2\E\left[R_K^2\right]^{1/2},
\]
uniformly in \(m\) and \(\boldsymbol{v}\). Letting \(K\to\infty\) proves uniform integrability of \(Y_i^2\).

To prove the required weak law, fix \(L>0\) and truncate \(Y_i^2\) at level \(L\). The variance of the average of the bounded, centered truncated variables is at most \(L^2/N\), so Chebyshev's inequality makes this average converge to zero in probability. Markov's inequality applied to the nonnegative tail average, together with uniform integrability, makes the contribution from \(Y_i^2>L\) uniformly small as \(L\to\infty\). Therefore,
\[
 \frac{1}{N}\sum_{i=1}^N Y_i^2\xrightarrow{\Pr}1.
\]
Finally,
\[
 Q_{N,m}
 =\frac{1}{m}\sum_{i=1}^N Y_i^2
 -\frac{1}{m}\left(\sum_{i=1}^N u_iY_i\right)^2.
\]
The first term converges in probability to \(c\) by \eqref{eq:dimension-ratio}. The second term is nonnegative and has expectation \(1/m\), because \(\sum_{i=1}^N u_iY_i\) has variance one. Markov's inequality therefore makes the second term converge to zero in probability. This proves \eqref{eq:q-convergence}.
\end{proof}

\begin{proof}[Proof of the sufficient part of Theorem~\ref{thm:exact-tail-condition}]
Let \(t_{N,m}:=c_{N,m}^{1/4}\). We prove \eqref{eq:uniform-upper-curve} by considering four ranges of \(\theta\). For every \(\theta\geq0\),
\[
 0\leq\mathcal{L}_{N,m}\left(\theta\right)\leq\theta^2.
\]
Hence, for \(0\leq\theta\leq t_{N,m}\),
\[
 \mathcal{L}_{N,m}\left(\theta\right)-\mathcal{R}_{c_{N,m}}\left(\theta\right)\leq0.
\]

We next consider values just above \(t_{N,m}\). For \(\theta>t_{N,m}\),
\[
 \theta^2-\mathcal{R}_{c_{N,m}}\left(\theta\right)
 =\frac{\theta^4-c_{N,m}}{c_{N,m}+\theta^2}.
\]
The right-hand side is zero at \(\theta=t_{N,m}\). Since \(c_{N,m}\to c\), the displayed expression is a continuous function of \(c_{N,m}\) and \(\theta\) on a fixed closed bounded rectangle containing these values. Therefore, for every \(\varepsilon>0\), one can choose \(\delta>0\) such that, for all sufficiently large \(N,m\),
\[
 \sup_{t_{N,m}\leq\theta\leq t_{N,m}+\delta}
 \left\{\mathcal{L}_{N,m}\left(\theta\right)-\mathcal{R}_{c_{N,m}}\left(\theta\right)\right\}
 \leq\varepsilon.
\]

For the remaining bounded values of \(\theta\), choose \(M>c^{1/4}+2\delta\). For all sufficiently large \(N,m\),
\[
 \left[t_{N,m}+\delta,M\right]
 \subset\left[c^{1/4}+\delta/2,M\right].
\]
Lemma~\ref{lem:bounded-uniformity}, applied to the fixed interval on the right, gives the required convergence on this range.

It remains to consider \(\theta\geq M\). Lemma~\ref{lem:fixed-theta} gives \(W\xrightarrow{\Pr}1+\sqrt{c}\). Lemma~\ref{lem:large-theta} therefore implies
\[
 \lim_{M\to\infty}\limsup_{N,m\to\infty}
 \Pr\left\{\sup_{\theta\geq M}\left|\mathcal{L}_{N,m}\left(\theta\right)-Q_{N,m}\right|>\varepsilon\right\}=0.
\]
Also, \eqref{eq:q-convergence} gives \(Q_{N,m}-c_{N,m}\xrightarrow{\Pr}0\), while a direct calculation gives
\[
 \sup_{\theta\geq M}\left|\mathcal{R}_{c_{N,m}}\left(\theta\right)-c_{N,m}\right|
 \leq\frac{C}{M^2}
\]
for all sufficiently large \(N,m\). Combining these four ranges gives the required upper bound. At \(\theta=0\), both \(\mathcal{L}_{N,m}\left(0\right)\) and \(\mathcal{R}_{c_{N,m}}\left(0\right)\) are zero, so the supremum in \eqref{eq:uniform-upper-curve} is nonnegative. This proves \eqref{eq:uniform-upper-curve}.

We now prove \eqref{eq:largest-error}. Since
\[
 \sup_{\theta\geq0}\mathcal{R}_{c_{N,m}}\left(\theta\right)
 =c_{N,m}\vee\sqrt{c_{N,m}},
\]
Equation~\eqref{eq:uniform-upper-curve} gives
\[
 \sup_{\theta\geq0}\mathcal{L}_{N,m}\left(\theta\right)
 \leq c\vee\sqrt{c}+o_{\Pr}\left(1\right).
\]
To obtain the matching lower bound, first suppose \(c<1\). Choose a fixed \(\theta>c^{1/4}\) sufficiently close to \(c^{1/4}\) that \(\mathcal{R}_c\left(\theta\right)\) is arbitrarily close to \(\sqrt{c}\). Lemma~\ref{lem:fixed-theta} gives the required lower bound. If \(c=1\), every fixed \(\theta>1\) satisfies \(\mathcal{R}_1\left(\theta\right)=1\). If \(c>1\), take any deterministic sequence \(\theta_{N,m}\to\infty\). Lemma~\ref{lem:large-theta} gives
\[
 \mathcal{L}_{N,m}\left(\theta_{N,m}\right)\xrightarrow{\Pr}c.
\]
This proves \eqref{eq:largest-error}.
\end{proof}

\subsection{When the tail condition fails}
\label{subsec:failure}

\begin{proof}[Proof of the converse statements in Theorem~\ref{thm:exact-tail-condition}]
Suppose that \eqref{eq:fourth-tail-condition} fails. Then there exist \(\delta_0>0\) and numbers \(x_k\to\infty\) such that
\[
 x_k^4\Pr\left\{\left|\Xi\right|>x_k\right\}\geq\delta_0.
\]
Take \(K>4\), to be enlarged below, and pass to a further subsequence if necessary so that
\[
 m_k:=\left\lfloor\left(\frac{x_k}{K}\right)^2\right\rfloor,
 \qquad
 N_k:=\left\lfloor cm_k\right\rfloor.
\]
Define
\[
 T_k:=\max_{\substack{2\leq i\leq N_k\\1\leq j\leq m_k}}
 \frac{\left|\Xi_{ij}\right|}{\sqrt{m_k}}.
\]
Since \(K\sqrt{m_k}\leq x_k\),
\[
 \Pr\left\{\left|\Xi\right|>K\sqrt{m_k}\right\}
 \geq\frac{\delta_0}{x_k^4}.
\]
There are \(\left(N_k-1\right)m_k\) entries in the maximum. If \(p_k:=\Pr\left\{\left|\Xi\right|>K\sqrt{m_k}\right\}\), independence gives
\[
 \Pr\left\{T_k\leq K\right\}=\left(1-p_k\right)^{\left(N_k-1\right)m_k}
 \leq\exp\left\{-\left(N_k-1\right)m_kp_k\right\}.
\]
Since \(m_k^2/x_k^4\to K^{-4}\) and \(\left(N_k-1\right)/m_k\to c\), the exponent is at least \(c\delta_0/\left(2K^4\right)\) for all sufficiently large \(k\). Hence
\[
 \liminf_{k\to\infty}\Pr\left\{T_k>K\right\}
 \geq1-\exp\left\{-\frac{c\delta_0}{2K^4}\right\}>0.
\]

Take \(\boldsymbol{u}=\boldsymbol{e}_1^{\left(N_k\right)}\) and \(\boldsymbol{v}=\boldsymbol{e}_1^{\left(m_k\right)}\). Let
\[
 R_k:=\left\|\left(\boldsymbol{e}_1^{\left(N_k\right)}\right)^{\top}\boldsymbol{X}\right\|
\]
be the norm of the first row of \(\boldsymbol{X}\). The weak law applied to \(\Xi^2\) gives \(R_k\xrightarrow{\Pr}1\). Moreover, \(R_k\) is independent of \(T_k\), because \(T_k\) uses only rows \(2,\ldots,N_k\).

Set \(\theta_0:=K/2\). Let \(S_k\) be the largest singular value of \(\boldsymbol{A}\left(\theta_0\right)\), and let \(\left(\boldsymbol{\hat{u}}_k,\boldsymbol{\hat{v}}_k\right)\) be a corresponding singular-vector pair. An entry attaining \(T_k\) lies outside the first row and is unchanged by the addition of \(\theta_0\boldsymbol{e}_1\boldsymbol{e}_1^{\top}\). Therefore, \(S_k\geq T_k\).

The first coordinate of \(\boldsymbol{A}\left(\theta_0\right)\boldsymbol{\hat{v}}_k=S_k\boldsymbol{\hat{u}}_k\) gives
\[
 S_k\left|\left(\boldsymbol{e}_1^{\left(N_k\right)}\right)^{\top}\boldsymbol{\hat{u}}_k\right|
 \leq\left\|\left(\boldsymbol{e}_1^{\left(N_k\right)}\right)^{\top}\boldsymbol{A}\left(\theta_0\right)\right\|
 \leq\theta_0+R_k.
\]
On the event \(\left\{T_k>K,\ R_k\leq K/4\right\}\),
\[
 \left|\left(\boldsymbol{e}_1^{\left(N_k\right)}\right)^{\top}\boldsymbol{\hat{u}}_k\right|
 \leq\frac{K}{2T_k}+\frac{R_k}{T_k}\leq\frac{3}{4}.
\]
Consequently,
\[
 \mathcal{L}_{N_k,m_k}\left(\theta_0\right)
 \geq\frac{K^2}{4}\left(1-\frac{9}{16}\right)
 =\frac{7K^2}{64}.
\]
Choose \(K\) sufficiently large that \(7K^2/64>c\vee\sqrt{c}+\varepsilon\) for some \(\varepsilon>0\). The positive lower bound for \(\Pr\left\{T_k>K\right\}\), the independence of \(T_k\) and \(R_k\), and the convergence \(R_k\xrightarrow{\Pr}1\) prove the claim for the fixed value \(\theta_0\).

Now suppose that the tail has the form stated in the last part of Theorem~\ref{thm:exact-tail-condition}. For every fixed \(K>0\),
\[
 \left(N-1\right)m\Pr\left\{\left|\Xi\right|>K\sqrt{m}\right\}
 =\left(N-1\right)m\left(K\sqrt{m}\right)^{-\alpha}\ell\left(K\sqrt{m}\right)
 \rightarrow\infty,
\]
because \(N/m\to c\), \(2-\alpha/2>0\), and \(\ell\) is slowly varying.
The same maximum argument gives \(T_{N,m}\xrightarrow{\Pr}\infty\), where \(T_{N,m}\) is defined as above with \(N,m\) in place of \(N_k,m_k\). The first-row norm satisfies \(R_{N,m}=O_{\Pr}\left(1\right)\), so \(R_{N,m}/T_{N,m}\xrightarrow{\Pr}0\).

For each realization of the matrix, the supremum over \(\theta\) is at least the value obtained at \(\theta=T_{N,m}/2\). The preceding coordinate bound then gives
\[
 \mathcal{L}_{N,m}\left(T_{N,m}/2\right)
 \geq\frac{T_{N,m}^2}{4}
 \left[1-\left(\frac{1}{2}+\frac{R_{N,m}}{T_{N,m}}\right)^2\right]
 \xrightarrow{\Pr}\infty.
\]
Therefore,
\[
 \sup_{\theta\geq0}\mathcal{L}_{N,m}\left(\theta\right)
 \xrightarrow{\Pr}\infty.
\]
\end{proof}

\subsection{A Gaussian consequence}
\label{subsec:gaussian-consequence}

Theorem~\ref{thm:exact-tail-condition} is stated after dividing the noise
matrix by \(\sqrt{m}\) and writing the added rank-one matrix as
\(\theta\boldsymbol{u}\boldsymbol{v}^{\mathsf T}\). The following corollary
gives the corresponding statement for a matrix in its original scale. The
first conclusion does not require the ratio \(N/m\) to converge.

\begin{corollary}
\label{cor:gaussian-matrix}
Let \(N\) and \(m\) tend to infinity in such a way that
\[
 \min\left\{N,m\right\}\rightarrow\infty,
 \qquad
 0<\underline{c}\leq\frac{N}{m}\leq\overline{c}<\infty.
\]
Let \(\boldsymbol{Z}\) be an \(N\times m\) matrix with independent standard
normal entries, and let \(\sigma>0\) be deterministic. For deterministic
vectors \(\boldsymbol{a}\in\R^N\) and \(\boldsymbol{b}\in\R^m\), define
\[
 \boldsymbol{M}\left(\boldsymbol{a},\boldsymbol{b}\right)
 :=
 \boldsymbol{a}\boldsymbol{b}^{\mathsf T}
 +
 \sigma\boldsymbol{Z}.
\]
Let
\(\boldsymbol{\hat{u}}\left(\boldsymbol{a},\boldsymbol{b}\right)\)
be a leading left singular vector of
\(\boldsymbol{M}\left(\boldsymbol{a},\boldsymbol{b}\right)\), chosen by the
same scale-invariant measurable rule as above, and define
\[
 \boldsymbol{\hat{P}}\left(\boldsymbol{a},\boldsymbol{b}\right)
 :=
 \boldsymbol{I}_N
 -
 \boldsymbol{\hat{u}}\left(\boldsymbol{a},\boldsymbol{b}\right)
 \boldsymbol{\hat{u}}\left(\boldsymbol{a},\boldsymbol{b}\right)^{\mathsf T}.
\]
Then, for every fixed \(\varepsilon>0\),
\begin{equation}
 \sup_{\substack{\boldsymbol{a}\in\R^N\\ \boldsymbol{b}\in\R^m}}
 \Pr\left\{
 \left\|
 \boldsymbol{\hat{P}}\left(\boldsymbol{a},\boldsymbol{b}\right)
 \boldsymbol{a}\boldsymbol{b}^{\mathsf T}
 \right\|_{\mathrm F}^2
 >
 \left(1+\varepsilon\right)\sigma^2
 \left(N\vee\sqrt{Nm}\right)
 \right\}
 \rightarrow0,
 \label{eq:gaussian-upper-bound}
\end{equation}
where \(\left\|\cdot\right\|_{\mathrm F}\) denotes the Frobenius norm.

Moreover, suppose that
\[
 \frac{N}{m}\rightarrow c\in\left(0,\infty\right).
\]
For every fixed \(\varepsilon\in\left(0,1\right)\), there exist deterministic
sequences
\(\boldsymbol{a}_{N,m}\in\R^N\) and
\(\boldsymbol{b}_{N,m}\in\R^m\) such that
\begin{equation}
 \Pr\left\{
 \left\|
 \boldsymbol{\hat{P}}\left(
 \boldsymbol{a}_{N,m},\boldsymbol{b}_{N,m}
 \right)
 \boldsymbol{a}_{N,m}\boldsymbol{b}_{N,m}^{\mathsf T}
 \right\|_{\mathrm F}^2
 >
 \left(1-\varepsilon\right)\sigma^2
 \left(N\vee\sqrt{Nm}\right)
 \right\}
 \rightarrow1.
 \label{eq:gaussian-lower-bound}
\end{equation}
Consequently, the number multiplying the right-hand side of
\eqref{eq:gaussian-upper-bound} cannot be replaced by a fixed number smaller
than one.
\end{corollary}

\begin{proof}
Suppose first that
\(\boldsymbol{a}\neq\boldsymbol{0}\) and
\(\boldsymbol{b}\neq\boldsymbol{0}\), and define
\[
 \boldsymbol{u}
 :=
 \frac{\boldsymbol{a}}{\left\|\boldsymbol{a}\right\|},
 \qquad
 \boldsymbol{v}
 :=
 \frac{\boldsymbol{b}}{\left\|\boldsymbol{b}\right\|},
 \qquad
 \theta
 :=
 \frac{
 \left\|\boldsymbol{a}\right\|
 \left\|\boldsymbol{b}\right\|
 }{
 \sigma\sqrt{m}
 }.
\]
Then \(\boldsymbol{u}\) and \(\boldsymbol{v}\) are unit vectors and
\[
 \frac{
 \boldsymbol{M}\left(\boldsymbol{a},\boldsymbol{b}\right)
 }{
 \sigma\sqrt{m}
 }
 =
 \frac{\boldsymbol{Z}}{\sqrt{m}}
 +
 \theta\boldsymbol{u}\boldsymbol{v}^{\mathsf T}.
\]
Multiplication by the positive number
\(\left(\sigma\sqrt{m}\right)^{-1}\) does not change the vector returned by the scale-invariant selection rule.
Therefore, with \(\mathcal{L}_{N,m}\) formed from these two unit vectors,
\[
 \begin{aligned}
 &
 \left\|
 \boldsymbol{\hat{P}}\left(\boldsymbol{a},\boldsymbol{b}\right)
 \boldsymbol{a}\boldsymbol{b}^{\mathsf T}
 \right\|_{\mathrm F}^2
 \\[3pt]
 &\qquad=
 \left\|
 \boldsymbol{\hat{P}}\left(\boldsymbol{a},\boldsymbol{b}\right)
 \boldsymbol{a}
 \right\|^2
 \left\|\boldsymbol{b}\right\|^2
 =
 \sigma^2m\,\mathcal{L}_{N,m}\left(\theta\right).
 \end{aligned}
\]
Also,
\[
 \sigma^2\left(N\vee\sqrt{Nm}\right)
 =
 \sigma^2m
 \left(
 c_{N,m}\vee\sqrt{c_{N,m}}
 \right).
\]
If either \(\boldsymbol{a}\) or \(\boldsymbol{b}\) is zero, the quantity on
the left of \eqref{eq:gaussian-upper-bound} is zero, so the same reduction
continues to hold after setting \(\theta=0\).

We now prove \eqref{eq:gaussian-upper-bound}. Suppose that it fails. Then
there exist \(\eta>0\), a subsequence of dimensions, and deterministic
vectors \(\boldsymbol{a}_{N,m}\) and \(\boldsymbol{b}_{N,m}\) such that
\[
 \Pr\left\{
 \mathcal{L}_{N,m}\left(\theta_{N,m}\right)
 >
 \left(1+\varepsilon\right)
 \left(
 c_{N,m}\vee\sqrt{c_{N,m}}
 \right)
 \right\}
 \geq\eta
\]
along that subsequence. Because \(c_{N,m}\) remains between
\(\underline{c}\) and \(\overline{c}\), a further subsequence satisfies
\[
 c_{N,m}\rightarrow c
 \in\left[\underline{c},\overline{c}\right].
\]
The standard normal distribution satisfies
\eqref{eq:fourth-tail-condition}. Theorem~\ref{thm:exact-tail-condition},
applied to the corresponding deterministic sequences of unit vectors, gives
\[
 \sup_{\vartheta\geq0}
 \mathcal{L}_{N,m}\left(\vartheta\right)
 \xrightarrow{\Pr}
 c\vee\sqrt{c}.
\]
Since
\[
 \left(1+\varepsilon\right)
 \left(
 c_{N,m}\vee\sqrt{c_{N,m}}
 \right)
 \rightarrow
 \left(1+\varepsilon\right)
 \left(
 c\vee\sqrt{c}
 \right),
\]
the preceding probability must tend to zero. This contradicts its lower
bound by \(\eta\) and proves \eqref{eq:gaussian-upper-bound}.

It remains to prove \eqref{eq:gaussian-lower-bound}. Suppose first that
\(c\leq1\). Choose a fixed \(\theta>c^{1/4}\) such that
\[
 \mathcal{R}_c\left(\theta\right)
 >
 \left(1-\frac{\varepsilon}{2}\right)\sqrt{c}.
\]
Such a choice is possible by taking \(\theta\) sufficiently close to
\(c^{1/4}\) when \(c<1\); when \(c=1\), every \(\theta>1\) satisfies
\(\mathcal{R}_1\left(\theta\right)=1\). By
Lemma~\ref{lem:fixed-theta},
\[
 \mathcal{L}_{N,m}\left(\theta\right)
 \xrightarrow{\Pr}
 \mathcal{R}_c\left(\theta\right).
\]
Since
\[
 \left(1-\varepsilon\right)
 \left(
 c_{N,m}\vee\sqrt{c_{N,m}}
 \right)
 \rightarrow
 \left(1-\varepsilon\right)\sqrt{c},
\]
we obtain
\[
 \Pr\left\{
 \mathcal{L}_{N,m}\left(\theta\right)
 >
 \left(1-\varepsilon\right)
 \left(
 c_{N,m}\vee\sqrt{c_{N,m}}
 \right)
 \right\}
 \rightarrow1.
\]

If \(c>1\), choose any deterministic sequence
\(\theta_{N,m}\to\infty\). Lemma~\ref{lem:large-theta} gives
\[
 \mathcal{L}_{N,m}\left(\theta_{N,m}\right)
 \xrightarrow{\Pr}c.
\]
At the same time,
\[
 \left(1-\varepsilon\right)
 \left(
 c_{N,m}\vee\sqrt{c_{N,m}}
 \right)
 \rightarrow
 \left(1-\varepsilon\right)c<c.
\]
Thus the same probability converges to one.

In either case, let
\[
 \boldsymbol{a}_{N,m}
 :=
 \sigma\sqrt{m}\,
 \theta_{N,m}\boldsymbol{e}_1^{\left(N\right)},
 \qquad
 \boldsymbol{b}_{N,m}
 :=
 \boldsymbol{e}_1^{\left(m\right)},
\]
where \(\theta_{N,m}=\theta\) when \(c\leq1\), and where
\(\theta_{N,m}\to\infty\) is the sequence chosen above when \(c>1\). The
identity proved at the beginning of the argument gives
\[
 \left\|
 \boldsymbol{\hat{P}}\left(
 \boldsymbol{a}_{N,m},\boldsymbol{b}_{N,m}
 \right)
 \boldsymbol{a}_{N,m}\boldsymbol{b}_{N,m}^{\mathsf T}
 \right\|_{\mathrm F}^2
 =
 \sigma^2m\,
 \mathcal{L}_{N,m}\left(\theta_{N,m}\right).
\]
This proves \eqref{eq:gaussian-lower-bound}.
\end{proof}

\section*{Statements and Declarations}

\noindent\textbf{Competing Interests.} The author declares no competing interests.

\medskip
\noindent\textbf{Data Availability.} No data were generated or analyzed in this theoretical study.

\end{document}